\documentclass[10pt,reqno]{amsart}

\usepackage{amsmath,amscd,amssymb }
\usepackage{cite}
\usepackage{enumerate}
\usepackage{mathtools}
\DeclarePairedDelimiter{\floor}{\lfloor}{\rfloor}

\newtheorem{proposition}{Proposition}[section]
\newtheorem{lemma}[proposition]{Lemma}
\newtheorem{theorem}[proposition]{Theorem}
\newtheorem{corollary}[proposition]{Corollary}
\newtheorem{conjecture}{Conjecture}[section]

\theoremstyle{definition}
\newtheorem{remark}[proposition]{Remark}
\newtheorem{definition}[proposition]{Definition}

\renewcommand{\O}{\Omega}

\renewcommand{\P}{\mathbb{P}}
\newcommand{\C}{\mathbb{C}}
\renewcommand{\O}{\mathcal{O}}

\DeclareMathOperator{\Aut}{Aut}

\DeclareMathOperator{\gr}{gr}

\DeclareMathOperator{\Rees}{Rees}
\DeclareMathOperator{\Proj}{Proj}
\DeclareMathOperator{\Span}{Span}
\DeclareMathOperator{\Sym}{Sym}

\newcommand{\N}{\mathbb{N}}

\newcommand{\Q}{\mathbb{Q}}

\newcommand{\pr}{\mathbb{P}}

\newcommand{\scF}{\mathcal{F}}
\newcommand{\scG}{\mathcal{G}}

\newcommand{\scL}{\mathcal{L}}

\newcommand{\scV}{\mathcal{V}}
\newcommand{\scX}{\mathcal{X}}
\newcommand{\scO}{\mathcal{O}}
\newcommand{\scR}{\mathcal{R}}

\renewcommand{\ll}{\ell \ell}

\DeclareMathOperator{\DF}{DF}

\DeclareMathOperator{\tr}{tr}

\title[Non-reductive automorphism groups and K-stability]{Non-reductive automorphism groups, the Loewy filtration and K-stability}
\author[G. Codogni]{Giulio Codogni}
\address{Giulio Codogni, Dipartimento di Matematica e Fisica, Universit\`{a} Roma Tre, Largo San Leonardo Murialdo 1, 00146, Roma, Italy.}
\email{codogni@mat.uniroma3.it}

\author[R. Dervan]{Ruadha\'i Dervan}
\address{Ruadha\'i Dervan, DPMMS, Centre for Mathematical Sciences, Wilberforce Road, Cambridge CB3 0WB, United Kingdom, and Universit\'e libre  de Bruxelles\\
1050 Elsene / Ixelles (Belgium).}
\email{R.Dervan@dpmms.cam.ac.uk}

\begin{document}

\begin{abstract}
We study the K-stability of a polarised variety with non-reductive automorphism group. We construct a canonical filtration, called the Loewy filtration, of the co-ordinate ring to each variety of this kind, which destabilises the variety in several examples which we compute. We conjecture that this holds in general. This is an algebro-geometric analogue of Matsushima's theorem regarding the non-existence of constant scalar curvature K\"ahler metrics on manifolds with non-reductive automorphism group. As an application, we give an example of an orbifold del Pezzo surface without a K\"ahler-Einstein metric. 
\end{abstract}

\maketitle


One of the most important problems in complex geometry is to understand the relationship between the existence of certain canonical metrics and algebro-geometric notions of stability. For vector bundles, the fundamental result is that a vector bundle admits a Hermite-Einstein metric if and only if it is slope polystable. Moreover, when the vector bundle is unstable, it admits a unique Harder-Narasimhan filtration of subsheaves such that each quotient is semistable. 

The analogous question for polarised varieties $(X,L)$ is the question of existence of constant scalar curvature K\"ahler (cscK) metric in the first Chern class of $L$. Here the notion of stability is K-polystability, and the Yau-Tian-Donaldson conjecture states that $(X,L)$ admits a cscK metric if and only if it is K-polystable \cite{SD3}. Loosely speaking, one assigns a certain weight (called the Donaldson-Futaki invariant) to each flat degeneration (called a test configuration) of $(X,L)$; K-semistability requires that this weight is non-negative for every test configuration. This is a weaker notion than K-polystability, and it is known, for example, that when $X$ is smooth that the existence of a cscK metric implies K-semistability \cite{SD}. The Yau-Tian-Donaldson conjecture has recently been proven when $X$ is Fano with $L=-K_X$, so that the metric is K\"ahler-Einstein \cite{RB,CDS,GT}.

Through the work of Witt Nystr\"om \cite{DWN} and Sz\'ekelyhidi \cite{GS,ICM}, one can interpret test configurations as \emph{admissible filtrations} of the co-ordinate ring. Roughly speaking, a filtration is called admissible if it is multiplicative and satisfies a linear boundedness condition. This is reviewed in Section \ref{Preliminars}. 

The approach to K-polystability via filtrations has at least two advantages. The first is conceptual: by including filtrations with non finitely generated Rees algebra, the notion of K-polystability is enhanced; this is discussed in \cite{ICM}. A classical example where non-finitely generated filtrations naturally occur is Zariski's example, as presented in \cite[Section 5.5]{Z}. The second is more practical: in some situations, it is easier to produce and describe examples of filtrations rather than of test configurations.

A natural geometric situation in which $(X,L)$ admits no cscK metric is when $X$ is smooth and the automorphism group $\Aut(X,L)$ is non-reductive \cite{YM}. Therefore, in this case, one would expect that such $(X,L)$ is not K-polystable. The goal of the present work is to construct a canonical filtration of the co-ordinate ring of a variety with non-reductive automorphism group. We call this filtration the \emph{Loewy filtration}; it is defined in Section \ref{loewy}.


\begin{theorem}\label{intro_main} Let $(X,L)$ be a polarised variety. Then, the Loewy filtration is admissible. Moreover, the Loewy filtration is $\Aut(X,L)$-equivariant, and is trivial if and only if $\Aut(X,L)$ is reductive. \end{theorem}

This result means that we can test K-polystability with this filtration. 


\begin{conjecture}\label{introconj} Suppose $(X,L)$ is a polarised variety with non-reductive automorphism group. Then the Loewy filtration destabilises $(X,L)$. \end{conjecture}

In Section \ref{examples}, we prove this conjecture in several cases.

\begin{theorem}\label{intro_examples}The Loewy filtration destabilises the following varieties: 
\begin{itemize}
\item[(i)] $\pr^2$ blown up at one point with respect to all polarisations.
\item[(ii)] $\pr^2$ blown up at n points on a line, with $L=a H - b(E_1 \hdots + E_n)$.
\item[(iii)] The Hirzebruch surfaces with respect to all polarisations.
\item[(iv)] Some projective bundles over $\pr^1$ with respect to all polarisations.
\item[(v)] A projective bundle over $\pr^2$ with respect to the anti-canonical polarisation.
\item[(vi)] An orbifold del Pezzo surface with respect to the anti-canonical polarisation.
\end{itemize}
\end{theorem}

It follows that none of the above polarised varieties admit constant scalar curvature K\"ahler metrics. 

A motivation for studying Conjecture \ref{introconj} is its relation to the proof of the Yau-Tian-Donaldson conjecture. One of the main technical steps in Chen-Donaldson-Sun's proof that a K-polystable Fano manifold admits a K\"ahler-Einstein metric was to show, using analytic methods, that if a Kawamata log terminal Fano variety admits a K\"ahler-Einstein metric, then its automorphism group is reductive \cite[Theorem 4]{CDS3}. It is known, however, that the existence of a K\"ahler-Einsten metric on such singular Fano varieties implies K-polystability \cite{RB}. A proof of Conjecture \ref{introconj} would therefore give an algebraic proof of Chen-Donladson-Sun's result. It is natural to expect that this would be important more generally in any attempt to show that K-polystability of a polarised variety implies the existence of a constant scalar curvature K\"ahler metric.

Another motivation for this work is to give a new method of destabilising varieties. All previous destabilising test configurations have either arisen from holomorphic vector fields, or used a particularly simple flat degeneration of the variety, namely deformation to the normal cone \cite{RT}. Stability with respect to such test configurations is called slope stability, and it is known that slope stability is strictly weaker to K-polystability. For example, the blow-up of $\pr^2$ at two points is slope stable, but admits no cscK metric \cite{PR}. Our method gives an algebro-geometric proof that the blow-up $\pr^2$ at two points is K-unstable.

\

\noindent {\bf Notation and conventions:} We often use the same letter to denote a divisor and the associated line bundle, and mix multiplicative and additive notation for line bundles. A polarised variety $(X,L)$ is a normal projective variety $X$ together with an ample line bundle $L$. Our results are independent of scaling $L$, as such we sometimes assume that $L$ is very ample and projectively normal.

\

\noindent {\bf Acknowledgements:} This project started during the school ``Minicourses on Stability'' at the University of Coimbra in April 2014; we thank the organisers for the stimulating environment. We would like to thank Giovanni Cerulli Irelli, Jesus Martinez Garcia, Julius Ross, Roberto Svaldi, Filippo Viviani and Xiaowei Wang for useful discussions. Both authors would especially like to thank David Witt Nystr\"om and Jacopo Stoppa for several discussions on the present work. We thank the referee for several helpful comments. 

The research leading to these results has received funding from the European Research Council under the European Union's Seventh Framework Programme (FP7/2007-2013) / ERC Grant agreement no. 307119. GC was funded by the grants FIRB 2012 ``Moduli Spaces and Their Applications'' and by the ERC StG 307119 - StabAGDG. RD was funded by a studentship associated to an EPSRC Career Acceleration Fellowship (EP/J002062/1) and a Fondation Wiener-Anspach Scholarship.

\section{Preliminaries on K-polystability}\label{Preliminars}

\begin{subsection}{Filtrations and K-semistability}
Let $(X,L)$ be a polarised variety. We are interested in the algebro-geometric concept of K-polystability; to define it, we associate to each filtration of the co-ordinate ring of $(X,L)$ a weight called the Donaldson-Futaki invariant. In this section we proceed in a purely algebraic way. In Section \ref{testconfigurations}, we will explain, following \cite{GS,ICM}, how to describe these concepts in a more geometric, and perhaps more familiar, language.

\begin{definition}\cite{DWN}\label{decreasingfiltration} 
Denote the  co-ordinate ring of $(X,L)$ by \begin{align*}R(X,L) &= \bigoplus_{k \geq 0} R_k, \\  &= \bigoplus_{k \geq 0} H^0(X,L^k).\end{align*}  We define an admissible decreasing filtration, or, for short, a \emph{decreasing filtration}, $\scF$ of $R$ to be sequence of vector subspaces 
$$\cdots \supset \scF_i R \supset \scF_{i+1} R \supset \cdots$$
which is
\begin{itemize}
\item[(i)] \emph{linearly right bounded}: there exists a constant $C$ such that $\scF_{Ck} R_k = \{0\}$ for every $k$, 
\item[(ii)] \emph{pointwise left bounded}: for every $k$ there exists a $j=j(k)$ such that $\scF_j R_k = H^0(X,L^k)$,
\item[(iii)] \emph{multiplicative}: $(\scF_i R_l) (\scF_j R_m) \subset \scF_{i+j} R_{l+m}$, 
\item[(iv)] \emph{homogeneous}: if $f\in \scF_iR$ then each homogeneous piece of $f$ is in $\scF_iR$; in other words $\scF_iR=\bigoplus_k\scF_iR_k$.
\end{itemize} 
Here we have denoted $\scF_iR_k=\scF_iR \cap R_k$.
\end{definition}

The Loewy filtration defined in Section \ref{loewy}, which is the main object of study in the present work, is an example of a decreasing filtration. Associated to each decreasing filtration are the following algebraic objects.

\begin{definition}Given a decreasing filtration $\scF$ on $R(X,L)$ we define its
\begin{itemize}
\item[(i)] \emph{Rees algebra} as $\Rees(\scF) = \bigoplus_{i}(\scF_iR)t^i \subset R[t],$
\item[(ii)] \emph{graded algebra} as $\gr(\scF) = \bigoplus_{i} (\scF_iR)/(\scF_{i+1}R).$
\end{itemize}
We say a filtration is \emph{finitely generated} if its Rees algebra is finitely generated as a $\C[t]$ module. Note that, in the definition of the graded algebra, the grade of $(\scF_iR)/(\scF_{i+1}R)$ is $i$.
\end{definition}

To each decreasing filtration, one can associate the following weight functions.

\begin{definition}\label{polynomials} Given a filtration $\scF$, define corresponding functions 
\begin{align*} 
 w(k) &= \sum_{i } i(\dim \scF_iR_k - \dim \scF_{i+1}R_k ), \\ 
 d(k) &= \sum_{i } i^2(\dim \scF_iR_k - \dim\scF_{i+1}R_k ). 
 \end{align*} 
We call $w(k)$ the \emph{weight function} and $d(k)$ the \emph{trace squared} function. Let $n$ be the dimension of $X$. We say that a filtration is \emph{polynomial} if $w(k)$ and $d(k)$ are polynomials of degree $n+1$ and $n+2$ respectively, for sufficiently large $k$. \end{definition}

Another invariant, which is associated just to the pair $(X,L)$ and is independent of the filtration, is the \emph{Hilbert function}
\begin{align*}h(k)= \dim R_k.\end{align*}
For $k$ sufficiently large this is always a polynomial of degree $n$.

Let $\scF$ be a polynomial filtration. Expand the associated Hilbert, weight and trace squared polynomials of $\scF$ respectively as \begin{align*}h(k) &=a_0k^n+a_1k^{n-1}+O(k^{n-2}), \\ w(k) &= b_0k^{n+1}+b_1k^n+O(k^{n-1}), \\ d(k) &= d_0 k^{n+2} + O(k^{n+1}).\end{align*}

\begin{definition}\label{donaldsonfutaki} We define the \emph{Donaldson-Futaki invariant} of a polynomial filtration to be $$\DF(\scF) = \frac{b_0a_1 - b_1 a_0}{a_0}.$$ We define the \emph{norm} of a polynomial filtration to be $$\|\scF\|_2 = \frac{d_0a_0 - b_0^2}{a_0}.$$\end{definition}

\begin{definition}\label{Ksemistable} We say that $(X,L)$ is \emph{K-semistable} if for all polynomial filtrations $\scF$, the Donaldson-Futaki invariant $\DF(\scF)$ is non-negative. If $(X,L)$ is not K-semistable we say that it is \emph{K-unstable}. 
\end{definition}

As proven in \cite{DWN,GS} and reviewed in Section \ref{testconfigurations}, finitely generated filtrations are equivalent to test configurations and the respective Donaldson-Futaki invariants are equal. In particular, all finitely generated filtrations are polynomial. On the other hand,  \cite[Example 4]{GS} is an example of a polynomial filtration which is not finitely generated. The coherence of our definition of the Donaldson-Futaki invariant and K-semistability for polynomials filtrations with the convention used in the literature is proved in Section \ref{SectionApproximation}, see in particular Theorem \ref{ApproximationTheorem}.

The conjecture that motivates this definition is the following. We give a precise definition of K-polystability in Section \ref{testconfigurations}, for the moment all we need is that K-polystability implies K-semistability.

\begin{conjecture}[Yau-Tian-Donaldson]\cite{SD2}\label{YTD} A smooth polarised variety $(X,L)$ is K-polystable if and only if $X$ admits a cscK metric in $c_1(L)$. \end{conjecture}

\begin{remark}\label{orbifold} By work of Donaldson, it is known that the existence of a cscK metric implies K-semistability \cite{SD}. Stoppa has strengthened this to K-polystability, provided the automorphism group of $(X,L)$ is discrete \cite{JS}. Therefore, the existence of a polynomial filtration with negative Donaldson-Futaki invariant implies that no cscK metric exists in $c_1(L)$. \end{remark}


While the filtration that we study in this paper naturally fits into the definition of a decreasing filtration, there is a another equivalent definition of filtrations which is more suitable to be translated into a geometric language. We will use this second notion to discuss the link between filtrations and test configurations, which are the more familiar object introduced by Donaldson \cite{SD3}.

\begin{definition}\cite{GS}\label{increasingfiltration} Denote by $R(X,L)$ the  co-ordinate ring of $(X,L)$. We define an admissible increasing filtration, or, for short, an \emph{increasing filtration}, $\scG$ of $R$ be sequence of vector subspaces 
$$\C=\scG_0R \subset \scG_1 R \subset \cdots$$
which is
\begin{itemize}
\item[(i)] \emph{pointwise right bounded}: for every $k$ there exists a $j=j(k)$ such that $\scG_j R_k = H^0(X,L^k)$,
\item[(ii)] \emph{multiplicative}: $(\scG_i R_l) (\scG_j R_m) \subset \scG_{i+j} R_{l+m}$, 
\item[(iii)] \emph{homogeneous}: $f\in \scG_iR$ then each homogeneous piece of $f$ is in $\scG_iR$.
\end{itemize} 
\end{definition}
Remark that the linear right bound in the decreasing case corresponds to the fact that the filtration starts at $0$ and $\scG_0=\C$. In this setup, the weight function is $$ w_{\scG}(k)= \sum_{i } (-i)(\dim \scG_iR_k  - \dim \scG_{i-1}R_k).$$


\begin{lemma} \label{swap}
There is a (non-canonical) way to pass from a decreasing polynomial filtration to an increasing polynomial filtration preserving both the Donaldson-Futaki invariant and the norm.
\end{lemma}

\begin{proof} Given a decreasing filtration $\scF$, we define $\scG_0=\C$ and $\scG_i(R_k)=\scF_{-i+Ck}(R_k)$, where $C$ is the constant appearing in the definition of decreasing filtration (remark that $C$ is not unique). It is easy to show that $\scG$ is point wise right bounded and homogenous; we now show that it is multiplicative. Indeed, 
\begin{align*}(\scG_{i_1} R_{k_1})( \scG_{i_2} R_{k_2} )&= (\scF_{Ck_1-i_1} R_{j_1})(\scF_{Ck_2-i_2} R_{j_2}), \\  &\subset \scF_{C(k_1+k_2)-i_1-i_2} R_{j_1+j_2}, \\ &= \scG_{i_1+i_2} R_{k_1+k_2}.\end{align*}
Remark that the linearity of the bound is key in this proof. To calculate the weight polynomials, note that we have added $Ck$ to the weight of each section of weight $i$. The weight polynomials are related by $$w_{\scG}(k) +Ckh(k)= w_{\scF}(k),$$ in particular $$b_{0,\scG} = b_{0,\scF} + Ca_0, b_{1,\scG} = b_{1,\scF} + Ca_1.$$
Similarly
$$d_{0,\scG}=d_{0,\scF}+2Cb_{0,\scF}+C^2a_0.$$
Computing the relevant Donaldson-Futaki invariants and norms we see that they are equal. For the reverse direction, it is enough to define $\scF_iR:=\scG_{-i}R$; all the verifications are straightforward, the constant $C$ can be taken equal to $0$.
\end{proof}

The procedure defined in the proof also gives an isomorphism between the Rees algebras of $\scF$ and $\scG$. This isomorphism does not preserve the grading, however finite generation is preserved.

\end{subsection}

\subsection{Approximating filtrations}\label{SectionApproximation} In this section, just to fix the notation, we assume that $\scF$ is an increasing filtration; however, in view of Lemma \ref{swap}, this does not really matter. The Donaldson-Futaki invariant of a non-finitely generated filtration $\scF$ is defined in \cite[Section 3.2]{GS} by choosing a specific approximation $\chi^{(k)}$ of $\scF$. It is not clear to us if this invariant depends on the choice of the approximation. A priori, for polynomial non-finitely generated filtrations, this definition does not coincide with ours. We avoid such issues by using the following result.

\begin{theorem}\label{ApproximationTheorem}
Let $\scF$ be a polynomial filtration. Then there exists a finitely generated filtration $\scG$ such that
$$ DF(\scF)=DF(\scG),$$
with the definition of the Donaldson-Futaki invariant as in Definition \ref{donaldsonfutaki}.
\end{theorem}

Theorem \ref{ApproximationTheorem} follows from the following two Lemmas.

\begin{lemma}\label{approximation}Let $\scF$ be an increasing filtration with Rees algebra $\scR$. For every integer $r$, there exists a finitely generated filtration $\scF^{(r)}$ of $R$ such that for all $p\leq r$ and all $i$ we have
$$\scF^{(r)}_iH^0(X,L^p) = \scF_i H^0(X,L^p).$$ \end{lemma}

\begin{proof}This is essentially contained in \cite[Section 3.1]{GS}. We construct the finitely generated filtration through its Rees algebra $\scR^{(r)}$, by defining $\scR^{(r)}\subset R[t]$ to be the $\C[t]$-subalgebra generated by 
$$\bigoplus^r_{p=1}\left(\bigoplus^{j(p)}_{i=1} (F_i R_p)t^i\right).$$
 Here $j(p)$ is the bound appearing in Definition \ref{increasingfiltration}. Since the filtration $\scF$ is multiplicative, this gives a well defined, finitely generated algebra. The corresponding filtration is 
$$\scF^{(r)}_i R = \{s\in R : t^i s \in \scR^{(r)}\}.$$
\end{proof}

\begin{remark} The geometric version of Lemma \ref{approximation} in terms of test configurations is \cite[Proposition 3.7]{RT}, which states that a test configuration is equivalent to an embedding of $X$ into projective space $\pr(H^0(X,L^k))$ for some $k$ and a choice of $\C^*$-action on this projective space. \end{remark}

\begin{lemma}\label{approximation2} Let $\scF$ be a polynomial filtration and $\scF^{(r)}$ be filtrations as in Lemma \ref{approximation}. Then, for $r$ sufficiently large, we have 
$$\DF(\scF^{(r)}) = \DF(\scF).$$ \end{lemma}

\begin{proof} Since $\scF$ is polynomial, its weight function is a polynomial of degree $n+1$.  Because of this, the weight polynomial is determined by finitely many values of $w(k)$. Take an approximating sequence as in Lemma \ref{approximation}, with $r$ sufficiently big such that the weight polynomial is determined by the weight values for $p<r$. Then $\scF^{(r)}$ has the same weight polynomial as $\scF$. Therefore $\DF(\scF^{(r)}) = \DF(\scF),$ as required. \end{proof}

To prove Theorem \ref{ApproximationTheorem} just remark that we can take as $\scG$ the filtration $\scF^{(r)}$ constructed in Lemma \ref{approximation} for $r$ sufficiently large and then apply Lemma \ref{approximation2}. Theorem \ref{ApproximationTheorem} implies the following.

\begin{corollary}
The following are equivalent:
\begin{itemize}
\item[(i)] For any filtration of the co-ordinate ring, the Donaldson-Futaki invariant is non-negative;
\item[(ii)] For any polynomial filtration of the co-ordinate ring, the Donaldson-Futaki invariant is non-negative;
\item[(iii)] For any finitely generated filtration of the co-ordinate ring, the Donaldson-Futaki invariant is non-negative.
\end{itemize}
\end{corollary}


This Corollary means that the definition of $K$-semistability \ref{Ksemistable} is equivalent to the usual definition. In particular, we can use Definition \ref{donaldsonfutaki} for the Donaldson-Futaki invariant of a polynomial filtration which is not finitely generated.

\subsection{Test configurations and filtrations}\label{testconfigurations}We now turn to the more geometric notion of test configurations, and recall how they relate to filtrations. 

\begin{definition} A \emph{test configuration} for $(X,L)$ is a polarised variety $(\scX,\scL)$ together with
\begin{itemize} 
\item[(i)] a proper flat morphism $\pi: \scX \to \C$,
\item[(ii)] a $\C^*$-action on $\scX$ covering the natural action on $\C$,
\item[(iii)] and an equivariant relatively ample line bundle $\scL$ on $\scX$
\end{itemize}
such that the fibre $(\scX_t,\scL_t)$ over $t$ is isomorphic to $(X,L^r)$ for one, and hence all, $t \in \C^*$ and for some $r>0$. 
\end{definition}

\begin{remark}One should think of test configuration as geometrisations of the one-parameter subgroups that are considered when applying the Hilbert-Mumford criterion to GIT stability on Hilbert schemes. \end{remark}

As the $\C^*$-action on $(\scX,\scL)$ fixes the central fibre $(\scX_0,\scL_0)$, there is an induced action on $H^0(\scX_0,\scL^k_0)$ for all $k$. Denote by $A_k$ the infinitesimal generator of this action, so that $\C^*$ acts as $t\to t^{A_k}$ on $H^0(\scX_0,\scL^k_0)$. The total weight $\tr(A_k)$ of the $\C^*$-action on $H^0(\scX_0,\scL^k_0)$ is a polynomial of degree $k+1$, expanding the Hilbert and weight polynomials as in Definition \ref{donaldsonfutaki} we can define the \emph{Donaldson-Futaki} invariant of a test configuration, just as we did for polynomial filtrations. Similarly, using  $\tr(A_k^2)$, one can define the \emph{norm} of a test configuration.

\begin{remark} There is a geometric interpretation of test configurations with zero norm: a test configuration $(\scX,\scL)$ has norm zero if and only if it has normalisation equivariantly isomorphic to the product configuration $X\times\C$ with the trivial action on $X$ \cite{RD,BHJ}.\end{remark}

One classical source of test configurations are those arising from automorphisms.

\begin{definition} Given a one parameter subgroups of $\Aut(X,L)$, define the corresponding \emph{product test configuration} by $(\scX,\scL)=(X\times\C,L)$ with the action on the central fibre over $0\in\C$ induced by the automorphism. \end{definition}

Given such an automorphism, Futaki showed that the corresponding Donaldson-Futaki invariant must vanish if $X$ admits a cscK metric in $c_1(L)$; see \cite{SD3} for a discussion of this point.

The relationship between test configurations and filtrations is as follows.

\begin{theorem}\cite{GS,DWN} Given an arbitrary test configuration, there exists a finitely generated filtration with the same Donaldson-Futaki invariant. Conversely, given any finitely generated filtration, one can construct a test configuration with the same Donaldson-Futaki invariant. \end{theorem}

\begin{proof} We recall the strategy of the proof.  Let $(\scX,\scL)$ be a test configuration, and let $s\in H^0(X,L^k)$. We think of $s$ as a section of $\scX_t$ for all $t \neq 0$ using the $\C^*$-action on $\scX$. In particular, $s$ can be thought of as a section defined when $t\neq 0$, so is a meromorphic section, with a pole of some order along $t=0$. Therefore $t^is$ is a holomorphic section for $i\in \N$ sufficiently large. We then define a filtration by saying that $s$ belongs to $\scF_i$ if $t^is$ is regular on all of $\scX$. 

Conversely, given a finitely generated filtration $\scF$, the associated test configuration is $\Proj_{\C[t]}(\Rees (\scF))$ with its natural $\scO(1)$ polarisation. The $\C^*$-action is given by the grading of the Rees algebra. \end{proof}

This theorem implies that finitely generated filtration are polynomials. With all of this in place, we can define K-polystability. 

\begin{definition}\cite[Section 8.2]{LX}\cite{JS2} We say that a polarised variety is \emph{K-polystable} if for all test configurations $(\scX,\scL)$, the corresponding Donaldson-Futaki invariant satisfies $\DF(\scX,\scL) \geq 0$, with equality if and only if $(\scX,\scL)$ is isomorphic to a product test configuration away from a closed subscheme of codimension two. Otherwise we say $(\scX,\scL)$ \emph{destabilises} $(X,L)$. \end{definition}

\begin{remark} The definition of K-polystability is independent of scaling $L\to L^r$. In particular, it makes sense for pairs $(X,L)$ where $X$ is a variety and $L$ is a $\Q$-line bundle. From another point of view, there is no loss in assuming that $L$ is very ample and projectively normal.\end{remark}

\section{The Loewy filtration}\label{loewy}

In this section we define the Loewy filtration and prove Theorem \ref{intro_main}. The Loewy filtration is a canonical decreasing filtration of the co-ordinate ring of any polarised variety $(X,L)$. The construction uses the automorphism group $\Aut(X,L)$ of $(X,L)$; the filtration is non-trivial if and only if $\Aut(X,L)$ is non-reductive. This filtration satisfies the hypotheses of Definition \ref{decreasingfiltration}: it is homogeneous and point-wise left bounded by construction, multiplicative because of Lemma \ref{mult} and linearly right bounded by Lemma \ref{lin_bound}. We assume that $L$ is very ample and projectively normal; in particular $\Aut(X,L)$ is a closed sub-group of $GL (H^0(X,L^k))$ for every $k$.

\begin{subsection}{The Loewy filtration of a module}

Let $G$ be a linear algebraic group defined over $\C$ and $V$ be a $G$-module. We are interested in the cases where $G$ is either $\Aut(X,L)$ or its unipotent radical and $V=H^0(X,L^k)$; however, our definition makes sense in a more general context. We define the filtration inductively.

\begin{definition}
Let $V$ be a finite dimensional $G$-module. The \emph{Loewy filtration} $\scF_{\bullet}V$ is defined as
\begin{itemize}
\item[(i)] $\scF_0V=V$;
\item[(ii)] for $i>0$, $\scF_iV$ is the minimal $G$-submodule of $\scF_{i-1}V$ such that the quotient $\scF_{i-1}V/\scF_iV$ is semi-simple.
\end{itemize}
\end{definition}

Equivalently, we can define $\scF_iV$ to be the intersection of all maximal non-trivial submodules of $\scF_{i-1}V$. The Loewy filtration is sometimes called radical filtration or Loewy decreasing filtration. See \cite[Section V.1]{ASS}  or \cite[Section I.2.14 and Chapter D]{JJ} for a general discussion. In what follows, we will give another description of the Loewy filtration. Before that, let us point out the following important consequence of the definition.

\begin{proposition} The Loewy filtration is $G$-equivariant; namely, each $F_iV$ is a $G$-submodule of $V$.

\end{proposition}

We wish to compute the weight of $\scF_{\bullet}$, as defined in Definition \ref{polynomials}. In order to do this more easily, we introduce another description of the Loewy filtration.
\begin{lemma}\label{unipotent_reduction}
Let $U$ be the unipotent radical of $G$ and $V$ be a $G$-module. Then the Loewy filtration of $V$ as a $G$-module is equal to the Loewy filtration of $V$ as a $U$-module. 
\end{lemma}
\begin{proof}
The unipotent radical of $G$ is its maximal normal connected unipotent subgroup; it is trivial if and only if $G$ is reductive. In characteristic zero, we have the Levi decomposition
$$
G=R\ltimes U
$$
where $R$ is a reductive group. A $U$-module is semi-simple if and only if the action of $U$ is trivial. Because of the normality of the radical, a maximal trivial $U$-submodule is also a $G$ submodule. Moreover it is a semi-simple $G$-submodule, because $U$ acts trivially and $R$ is reductive.
\end{proof}
Lemma \ref{unipotent_reduction} simplifies the study of the Loewy filtration because a representation of a unipotent group is semi-simple if and only if the action of the group is trivial.

\end{subsection}
\begin{subsection}{The Loewy filtration of the co-ordinate ring}

We now consider the Loewy filtration of the co-ordinate ring. Let $(X,L)$ be a polarised variety and let
$$
R=\bigoplus_{k\geq 0}R_k=\bigoplus_{k\geq 0}H^0(X,L^{\otimes k})
$$
Each module $R_k$ has a Loewy filtration $\scF_{\bullet}$; we define
$$ \scF_iR:=\bigoplus_k\scF_iR_k.$$

\begin{lemma}\label{mult} The Loewy filtration is multiplicative. \end{lemma}

\begin{proof}
We use a further description of the Loewy filtration. Let $\mathfrak{u}$ be the Lie algebra of $U$ and $A$ be the universal enveloping algebra $\mathfrak{U}(\mathfrak{u})$ of $\mathfrak{u}$. The advantage of this point of view is that $A$ is an associative (non-commutative) algebra. We can thus consider its Jacobson radical $J(A)$, see for example \cite[Section I.3]{ASS}. It can be defined as the intersection of all maximal left ideals and one can show that it is a two-sided ideal. The relationship between the Loewy filtration and the Jacobson radical is 
$$ \scF_iR_k=J(A)^iR_k,$$
this follows from \cite[Proposition I.3.7 and Corollary I.3.8]{ASS}. From this we can show that the Loewy filtration is multiplicative. Indeed,

\begin{align*}\scF_iR_l\cdot \scF_jR_s&=J(A)^iR_l\cdot J(A)^jR_s, \\ &\subset J(A)^{i+j}R_{l+s}, \\ &=\scF_{i+j}R_{l+s}.\end{align*}
\end{proof}
\begin{definition}
We define the \emph{length} of a filtration to be the maximum $i$ such that $F_iV$ is not trivial. We then define the \emph{Loewy length} $\ll(V)$ of a module $V$ is the length of the Loewy filtration. The Loewy length of a graded ring 
$$R=\bigoplus_k R_k $$
is a function of $k$, and we denote by $\ll_R(k)$ the Loewy length of $R_k$. The Loewy length of a polarised variety $(X,L)$ is the length of its co-ordinate ring, seen as an $\Aut(X,L)$-module.
\end{definition}
General discussions about the Loewy length of a fixed module can be found in \cite[Chapter V]{ASS} and \cite[Chapter D]{JJ}.

\begin{lemma}\label{lin_bound}
The Loewy filtration is linearly right bounded. That is, there exists a constant $C$ such that
$$\ll(k)\leq Ck.$$
We can take $C=h^0(X,L)$.
\end{lemma}
\begin{proof}
Let $V:=H^0(X,L)$. Since we have assumed that $L$ is projectively normal, we have a surjective map
$$\Sym^k V\to H^0(X,L^k).$$
The Loewy length of the domain is bigger than the Loewy length of the codomain (cf. \cite[Proposition I.3.7]{ASS}), so it is enough to prove the statement for $\Sym^k V$. The unipotent radical $U$ of $\Aut(X,L)$ is a subgroup of some maximal unipotent subgroup $T$ of $GL(V)$. The Loewy length of $\Sym^k V$ as $\Aut(X,L)$-module is smaller than its Loewy length as a $T$-module. Below we will show that the Loewy length of $\Sym^k V$ as a $T$ module is $k(\dim V -1)+1$, and this is enough to conclude the proof.

To compute the Loewy length of $\Sym^k V$ as a $T$ module fix a basis $e_i$ for which $T$ is the group of upper triangular matrices. Assign weight $n-i$ to $e_i$, where $n=\dim V$. Each monomial in $\Sym^k V$ now has a weight and there are exactly $k(n-1)+1$ different weights. The point is that this weight corresponds to the grade assigned by the Loewy filtration because $T$ can only increase the weight.
\end{proof}
Remark that in our examples the Loewy length is linear. It is an interesting question to ask for which class of varieties this is true; moreover, when the Loewy length \emph{is} a linear function, we do not know if its slope has a geometric meaning.

\begin{conjecture}\label{main-conj}Let $(X,L)$ be a polarised variety with non-reductive automorphism group. Then the Loewy filtration is polynomial, and destabilises $(X,L)$. In particular, $(X,L)$ is not K-polystable.\end{conjecture}
By destabilises we mean either that it has strictly negative Donaldson-Futaki invariant (which is the case in all the examples), or that the Donaldson-Futaki invariant is zero and one of the corresponding finitely generated filtrations that we can associate to $\scF$ via Theorem \ref{ApproximationTheorem} is not isomorphic to a product test configuration away from a closed subscheme of codimension two.

\end{subsection}

\begin{subsection}{Vanishing order filtration}\label{vanishing_order}

So far, we have just used representation theory and the geometry of $X$ never appeared. Motivated by understanding the geometry of the Loewy filtration, we introduce another related geometric filtration of the co-ordinate ring of $(X,L)$. We will use this filtration in Propositions \ref{fg} and \ref{fg2} to prove, in some cases, that the Loewy filtration is finitely generated.

Let $U$ be the unipotent radical of the automorphism group $\Aut(X,L)$, and let $E:=X^U$ be the fixed sub-scheme of the action of $U$ on $X$. There are a few general results about the geometry of $E$. The key one is Borel's fixed point theorem, which guarantees that $E$ is non-empty. Horrocks showed that $E$ is connected \cite{GH}. For related work, see \cite{Gross} when $X$ is projective and \cite{JL} for $X$ affine. We can define a multiplicative decreasing filtration on $R$ by the vanishing order along $E$, namely
$$\scV_iH^0(X,L^k):=H^0(X,L^k(-iE)).$$
We do not have a general bound on the length of this filtration. In all our examples $X^U$ is a reduced divisor; we do not know how general this fact is. In some cases, such as the Hirzebruch surfaces, this filtration equals the Loewy filtration; in other cases, such as $\P^2$ blown up at two points, the filtration by vanishing order is strictly included in the Loewy filtration.

\end{subsection}

\begin{section}{Examples}\label{examples}
In this section we describe the Loewy filtration of some varieties, giving a proof of Theorem \ref{intro_examples}. In all examples the filtration is polynomial, so we do not need to take approximations to compute the Donaldson-Futaki invariant. In all of our examples, the Donaldson-Futaki invariant is negative, confirming Conjecture \ref{main-conjecture} in these cases. Some examples are special cases of others, we include them for the sake of clarity. We will use Lemma \ref{unipotent_reduction} to compute the Loewy filtration.

\begin{subsection}{Degree 8 del Pezzo}\label{deg8}
Let $X$ be the blow-up of $\P^2$ at a point $p$. Fix an ample line bundle $ L=aH-bE$, recall that the ampleness is equivalent to $a>b>0$. Fix a basis $x,y,z$ of $H^0(\P^2, \O(1))$ such that $p=[1,0,0]$. In this basis, we have an identification
\begin{align*}R_k&:= H^0(X,kL), \\&= \Span \{ \mbox{degree } ka \mbox{ monomials such that }  \deg y + \deg z \geq bk \}.\end{align*}
The Hilbert polynomial is
\begin{align*}h(k)&=\sum_{i=0}^{ak-bk}(ak-i+1), \\ &=\frac{1}{2}(a^2-b^2)k^2+\frac{1}{2}(3a-b)k+1,\end{align*} which can also be seen from Riemann-Roch. We now describe the Loewy filtration. The automorphism group $G$ is the sub-group of $\P GL(3)$ fixing $p$; its unipotent radical $U$ is
\begin{displaymath}
U=
\left(
\begin{array}{ccc}
1 & * & * \\
0 & 1 & 0 \\
0 & 0  & 1 
\end{array}
\right).
\end{displaymath}
Recall that, in view of Lemma \ref{unipotent_reduction}, it is enough to study the action of $U$. On the space of sections we have the dual action, so this group fixes $y$ and $z$. This means that $\scF_1R_k$ is generated by all monomials for which $x$ does not appear with maximal degree; indeed, in this way $\scF_1R_k$ is a submodule and quotienting gives a trivial $U$-module. Remark that the maximal degree of $x$ in $R_k$ is $ak-bk$. More generally, we have 
\begin{align*}\scF_i(R_k) = \Span \{& \mbox{degree } ka \mbox{ monomials such that } \\ & \deg y + \deg z \geq bk \mbox{ and } \deg x \leq ak-bk-i\}.\end{align*}
For the associated graded modules we have 
\begin{align*} \gr_i(R_k)=\Span \{& \mbox{degree } ka \mbox{ monomials such that } \\ & \deg y + \deg z \geq bk \mbox{ and } \deg x=ak-bk-i\}.\end{align*}
As an example, for $a=3, b=1, k=1$, the graded module associated to the Loewy filtration of $H^0(X,-K_X)$ is
$$ \gr_0 = \Span \{x^2y,x^2z\}, \quad \gr_1 = \Span \{xy^2,xyz,xz^2\}, \quad\gr_2=\Span \{y^3,z^3, y^2z, yz^2 \}.$$

We wish to count the dimension of the weight space of weight $i$. Since $\deg x = ak-bk-i$, we have $\deg y + \deg z =bk+ i$. There are $bk+i+1$ polynomials in two variables of degree $bk+i$. Therefore the dimension of the weight space is $bk+i+1$. The weight polynomial is
\begin{align*}w(k)&=\sum_{i=0}^{ak-bk}i(bk+i+1), \\ &= \left(\frac{1}{3}a^3-\frac{1}{2}a^2b+\frac{1}{6}b^3\right)k^3+\left(a^2-\frac{3}{2}ab+\frac{1}{2}b^2\right)k^2+\left(\frac{2}{3}a-\frac{2}{3}b\right)k. \end{align*}

The numerator of the Donaldson-Futaki invariant is 
$$\DF_{num}=-\frac{1}{6}b^4\left(\frac{a}{b}-1\right)^3.$$ 
This is negative \emph{exactly} when $a>b$, which is required for ampleness. The Loewy length is $\ll(k) = (a-b)k$. The norm is 
$$\|\scF\|_2=\frac{a^4}{4}-\frac{2(a^3-b^3)^2}{9(a^2-b^2)},$$ 
which is positive for $a>b$.  For $a=b$ the Donaldson-Futaki invariant vanishes. This is not surprising because, in this case, we are dealing with a line bundle which is a bull-back from $\P^1$, which is K-polystable.

\end{subsection}

\begin{subsection}{$\pr^2$ blown up at $n$ points on a line}

Let $X$ be the blowup of $\P^2$ at $n$ points $p_1,\hdots,p_n$ on a line, with $p_1 = [1:0:0]$ and $p_2=[0:1:0]$. The picard rank is $\rho(X) = n+1$, generated by $H$, which is the pullback of the hyperplane class from $\P^2$, and the $n$ exceptional divisors $E_1,\hdots,E_n$. To check a line bundle $L$ is ample on $X$, it suffices to show it has positive intersection with $H-E_1-\hdots-E_n$ and the exceptional divisors $E_1,\hdots,E_n$. The ample cone is therefore those $L=aH-b_1E_1-\hdots-b_nE_n$ such that $a>b_1+\hdots+b_n$, with $a,b_1,\hdots,b_n>0$. For simplicity for the rest of this calculation we assume $b_1=\hdots=b_n=:b$, so that the condition for ampleness becomes $a>nb$. We also assume that $n\geq 2$, having considered the $n=1$ case in Section \ref{deg8}. Denote $L=aH-bE_1-\hdots-bE_n$, and $R_k = H^0(X,kL)$. For ease of notation denote $E = E_1+\hdots+E_n$. $E$ has the property that $E.E = -n$, $L.E = nb$. 

The anti-canonical class of $X$ is $-K_X = 3H-E$, which intersects $L$ as $-K_X.L = 3a-nb$. $L$ has self-intersection $L.L = a^2-nb^2>0$. By Riemann-Roch the Hilbert polynomial is given as $$h(k) = \frac{a^2-nb^2}{2}k^2 + \frac{3a-nb}{2}+1.$$

The automorphism group of $X$ is the subgroup of $PGL(3)$ consisting of $3\times 3$ matrices which either fix the line $[a:b:0]$ or permute some of the $n$ points on the line. The permutations do not lie in the connected component of the identity, so the maximal normal unipotent subgroup of $\Aut(X)$, which we denote by $U$, is given by matrices of the form \begin{displaymath}
U=
\left(
\begin{array}{ccc}
1 & 0 & * \\
0 & 1 & * \\
0 & 0  & 1 
\end{array}
\right).
\end{displaymath} Note that matrices of this form fix the line which joins the blown-up points. 

The space $R_k$ is the space of degree $ak$ polynomials which vanish with order at least $kb$ at each point $p_i$; therefore it contains all polynomials whose degree in $z$ is at least $bk$, and some of the others. The action of $U$ on $H^0(\P^2,\O(1))$ fixes exactly $z$, so the Loewy filtration is
$$\scF_iR_k = \{  \mbox{monomials in } R_k \mbox{ such that } \deg z \geq i \}.$$

For $i\geq kb$, the Loewy filtration is 
$$\scF_iR_k = \{ \deg ka \mbox{ monomials with } \deg z \geq i \}.$$
 Thus, for $i\geq kb$, we have $\dim \scF_iR_k - \dim \scF_{i-1}R_k = ka-i+1$. 

We now calculate the dimensions of the spaces $\scF_iR_k$ for $i < kb$. Fixing some $\deg z=i$, we wish to count the number of degree $ka-i$ polynomials in $2$ variables which vanish along $n$ points on a line with order at least $kb-i$ at each point $p_i$. Such polynomials have weight $i$, and vanish order at least $kb$ along $n$ points. The number of such polynomials is given by Riemann-Roch on $\P^1$. Indeed, let $$M=(ka-i)\O_{\P^1}(1) - (kb-i)p_1 - \hdots - (kb-i)p_n,$$ where we consider those points to be in $\P^1$. By Riemann-Roch for curves we have \begin{align*}\dim H^0(\P^1, M) &= \deg M +1, \\ & = ka-i - n(kb-i) +1, \\&= k(a-nb)+(n-1)i+1.\end{align*} So for $i\leq kb$ we have $$\dim \scF_iR_k - \dim \scF_{i-1}R_k = k(a-nb)+(n-1)i+1.$$ The Loewy length is $\ell \ell(k)=ak$. Using this method, the Hilbert polynomial is $$h(k)=\sum^{kb-1}_{i=0} (k(a-nb)+(n-1)i+1) + \sum^{ka}_{i=kb}(ka-i+1),$$ which agrees with the calculation using Riemann-Roch. Note that when $i=kb$ the two summands are equal.  The weight polynomial is $$w(k)=\sum^{kb-1}_{i=0} i(k(a-nb)+(n-1)i+1) + \sum^{ka}_{i=kb}i(ka-i+1).$$ 

The weight polynomial has highest terms $$b_0 = \frac{a^3-nb^3}{6},\qquad b_1 = \frac{a^2}{2}.$$ The numerator of the Donaldson-Futaki invariant is \begin{align*}12\DF(\scF)_{num}&=12(b_0a_1-b_1a_0), \\ & = (a^3-nb^3)(3a-bn) - 3(a^2)(a^2-nb^2).\end{align*} Expanding and setting $c=\frac{a}{b}$ we get $$\frac{12}{b^4}\DF(\scF)_{num}=-nc^3+3nc^2-3nc+n^2.$$ The condition $a>nb$ becomes $c>n\geq 2$. Dividing by $n$ we wish to show $$-c^3+3c^2-3c+n<0.$$ Note that when $n=c$ this is given as $c(c(3-c)-2),$ which is less than or equal to zero when $c\geq 2$. Its derivative is $-3(c-1)^2$, which is negative when $c>1$. In particular, this is a decreasing polynomial in $c$ when $c>1$, which is negative when $c=2$. So it is negative for all $c>2$, therefore this filtration destabilises. The norm is given as $$\|F\|_2 = \frac{1}{36}\left(3(a^4-b^4n) - 2\frac{a^3-nb^3}{a^2-nb^2}\right),$$ setting $c=\frac{a}{b}$ and using $c>n\geq 2$ and $b>0$ one sees that this is strictly positive.

\end{subsection}
\begin{subsection}{An orbifold del Pezzo surface}\label{orbifold}
This example is based on the analysis developed in the previous section, hence we will keep the same notations. Let $\mu\colon X\to \P^2$ be the blow up of $\P^2$ at $3$ points $p_i$ on a line $\ell$. Let $\hat{\ell}$ be the proper transform of $\ell$, its class is $H-E$. The divisor $\hat{\ell}$ is a $-2$ curve so we can consider its blow-down
$$\nu\colon X\to F.$$ 
The variety $F$ is a singular Fano; the singular point is an $A_1$ singularity so the surface is an orbifold. The singularity is rational, so $\nu^*K_F=K_X$ (this can be shown also by explicit computation). For any line bundle $L$ on $F$, the pull-back defines an isomorphism
$$
\nu^*\colon H^0(F,L)\to H^0(X,\nu^*L).
$$
Since $\hat{\ell}$ is fixed by $U$, the isomorphism is an isomorphism of $U$-modules, so it preserves the Loewy filtration. This means that the Loewy filtration, and its Donaldson-Futaki invariant, can be equivalently computed on $X$ or on $F$. Line bundles on $X$ which are of the form $\nu^*L$ are the ones with zero intersection with $\hat{\ell}$, therefore we are interested in multiples of $3H-E$, which is actually the anti-canonical class of $X$. Plugging $c=n=3$ in the formula for the numerator of the Donaldson-Futaki invariant obtained in the previous section, we see that the Loewy filtration destabilises $(F,-K_F)$.

\end{subsection}

\begin{subsection}{Hirzebruch surfaces}
We consider the Hirzerburch surface
$$X=\P ( \O\oplus \O(n)).$$
This is a $\P^1$-bundle over $\P^1$. Denote by $H$ the pull-back of the hyperplane from $\P^1$ and $\O(1)$ the tautological line bundle. Let $$L=a\O(1)+bH,$$ this is ample if and only if $a,b>0$. Pushing forward to $\P^1$, we have
\begin{align*}R_k:&= H^0(X,kL), \\&=\bigoplus_{i=0}^{ka}H^0(\P^1,\O_{\pr^1}(bk+in)).\end{align*}
The Hilbert polynomial is
\begin{align*}h(k)&=\sum_{i=0}^{ak}(bk+in+1),\\ &=k^2(\frac{1}{2}a^2n+ab)+k(\frac{1}{2}an+a+b)+1.\end{align*}

Let us describe the Loewy filtration. The unipotent radical of the automorphism group is $H^0(\P^1,\O(n))$; it acts on the total space of $\O\oplus\O(n)$ as the upper triangular matrices \cite[Section 5.11]{Debarre}\cite{Levine}. That is, a section $s$ of $H^0(\P^1,\O(n))$ maps an element $c\oplus 0$ of $\O\oplus \O(n)$ to $c\oplus s$. The induced action on $R_k$ maps $H^0(\P^1,\O(bk+in))$ to $H^0(\P^1,\O(bk+in))\oplus H^0(\P^1,\O(bk+(i+1)n))$. The Loewy filtration is thus
$$\scF_iH_k=\bigoplus_{j=i}^{ka}H^0(\P^1,\O(bk+jn)).$$
The graded modules are
$$
\gr_i=H^0(\P^1,\O(bk+in)).
$$
The weight polynomial is
$$w(k)=\sum_{i=0}^{ak}i(bk+in+1),$$
so that
$$b_0=\frac{1}{3}a^3n+\frac{1}{2}a^2b,\qquad b_1=\frac{1}{2}a^2n+\frac{1}{2}a^2+\frac{1}{2}ab. $$
The numerator of the Donaldson-Futaki invariant is
$$ \DF_{num}=-\frac{1}{12}a^4n^2+\frac{1}{12}a^4n-\frac{1}{6}a^3bn,   $$ 
which is negative. The Loewy length is $\ll(k) = ak$. The norm is 
$$\|\scF\|_2=\frac{a^3(a^2n^2+6abn+6b^2)}{36(an+2b)},$$
which is positive for $a,b>0$.
Remark that for $n=1$ the Hirzebruch surface is isomorphic to the blow up at $\P^2$ at one point, and the Loewy filtrations using both descriptions coincide. This can be checked by an explicit computation. Indeed, to compare the two descriptions take $b=b'$ and $a=a'+b'$, where $a$ and $b$ are the parameters appearing in the del Pezzo description, and $a'$ and $b'$ are the parameters appearing in the projective bundle description. 

We now use the filtration defined in Section \ref{vanishing_order} to show that the Loewy filtration is finitely generated.
\begin{proposition}\label{fg}
The Loewy filtration of a Hirzebruch  surface equals the filtration by vanishing order along the fixed locus. Moreover, it is finitely generated.
\end{proposition}

\begin{proof}
The fixed locus of the action of $U$ on $X$ is the $-n$ curve $E$. This curve has class $E=\O(1)-nH$. We have
\begin{align*} \scV_iR_k&=H^0(X,ka\O(1)+kbH-iE), \\&=H^0(X,(ak-i)\O(1)+(kb+ni)H).\end{align*} 
Pushing forward to $\P^1$ we get
\begin{align*} H^0(\P^1,\O(kb+in)\otimes \Sym^{ak-i}(\O\oplus\O(n)))&=\bigoplus_{j=0}^{ak-i}H^0(\P^1,\O(kb+(i+j)n)), \\ &=\bigoplus_{j=i}^{ak}H^0(\P^1,\O(kb+jn)).\end{align*}
This coincides with the Loewy filtration. The Rees algebra associate to this filtration is isomorphic to
$$\bigoplus_{d,m}H^0(X,dL-mE).$$

We claim this ring is finitely generated. This is a standard argument that follows from the Minimal Model Program, using that $X$ is a toric variety, hence a log Fano variety. Set $F$ to be an integral divisor lying on the boundary of the effective cone of $X$ which is linearly equivalent to $aL-bE$, with $a,b$ positive integers. $F$ is automatically effective as the effective cone of a log Fano variety is rational polyhedral, hence closed \cite{BCHM}. Using \cite[Lemma 2.3.3]{CortiFlips}, it is enough to show the ring $$R(X,L,F)= \bigoplus_{m',d'\in \N^2}H^0(X,d'L+m'F)$$ is finitely generated. Since $X$ is toric, there exists an ample $\Q$-divisor $A$ and an effective $\Q$-divisor $\hat{E}$ with $$K_X+A+\hat E\sim 0.$$ Set $\Delta_1=A+\hat E+\frac{1}{n_1}L$ and $\Delta_2 = A+\hat E+\frac{1}{n_2}F$ with $n_1,n_2\gg 0$ chosen so that $D_i=K_X+\Delta_i$ are divisorially log terminal divisors for $i=1,2$. Remark that $D_1\sim \frac{1}{n_1}L$, and $D_2\sim \frac{1}{n_2}F$. By \cite[Corollary 1.1.9]{BCHM}, the ring $$\bigoplus_{m'',d''\in \N^2}H^0(X,\floor{d''D_1+m''D_2})$$ is finitely generated. Using again \cite[Lemma 2.3.3]{CortiFlips}, this is equivalent to finite generation of $R(X,L,F)$, as required.

\end{proof}
Following \cite[Section 8]{DWN}, the test configuration associated to this filtration is the deformation to the normal cone of the $-n$-curve with parameter $a$. To identify the two filtrations we used that Loewy length is linear.

\end{subsection}

\begin{subsection}{Projective bundles $\P (\O^{\oplus r } \oplus \O(n))$ over $\P^s$ - partial computation}\label{GeneralProjectiveBundle}

This is a generalisation of the Hirzerbrch surface example. Let $X=\P (\O^{\oplus r} \oplus \O(n))$ over $\P^s$. Take as polarisation $L=a \O(1)+bH$, where $H$ is the pull-back of the hyperplane section of $\P^s$. We take $a$,$b$ and $n$ strictly positive. Pushing-forward we have
$$H^0(X,kL)=H^0(\P^s,\O(bk)\otimes \Sym^{ka}(\O^{\oplus r}\oplus \O(n))).$$
The right hand side is isomorphic to
$$  \bigoplus _{J}H^0(\P^s,\O(bk+J_{r+1}n)),$$
where the sum runs over all partitions $J=(J_1,\cdots , J_{r+1})$ of $ak$ into $r+1$ non-negative numbers. Here we are just writing out monomials in $r+1$ variables, $J_{r+1}$ is the exponent of $\O(n)$. The unipotent radical of the automorphism group is $H^0(\P^1,\O(n))^{\oplus r}$, see \cite[Section 5.11]{Debarre}. It maps $\O^{\oplus r}$ to $\O^{\oplus r}\oplus \O(n)$. The graded module associated filtration induced by the action is
$$ \gr_i=H^0(\P^s,\O(bk+in))^{\oplus P(ak-i,r)},$$
where $P(ak-i,r)$ is the number of partition of $ak-i$ into $r$ non-negative numbers. The index $i$ ranges from $0$ to $ak$. That is, $\gr_i$ is isomorphic to the sub-vector space where $\O(n)$ appears with multiplicity exactly $i$.  The Loewy length is $\ell \ell(k)=ak$. The Hilbert polynomial is
$$ h(k)=\sum_{i=0}^{ak}\binom{bk+in+s}{s}P(ak-i,r).$$
The weight polynomial is
$$ w(k)=\sum_{i=0}^{ak}i\binom{bk+in+s}{s}P(ak-i,r).$$

Using the vanishing order filtration as defined in Section \ref{vanishing_order}, we now show that the Loewy filtration is finitely generated.
\begin{proposition}\label{fg2}
The Loewy filtration equals the filtration by vanishing order along the fixed locus. Moreover, it is finitely generated.
\end{proposition}
\begin{proof}
The fixed locus $E$ of the action of $H^0(\P^1,\O(n))^{\oplus r}$ is the reduced divisor corresponding to the quotient
$$
\pi \colon \O^{\oplus r}\oplus \O(n)\to \O^{\oplus r}
$$
because $\O(n)$ is the maximal sub-vector bundle of $\O^{\oplus r}\oplus \O(n)$ which is $H^0(\P^1,\O(n))^{\oplus r}$ invariant. A section of $kL$ vanishes along $E$ if and only if it is in the kernel of $\Sym^k (\pi)$, so the Loewy filtration equals the filtration by vanishing order along $E$. The Rees algebra of this filtration is isomorphic to
$$ \bigoplus_{m,d\in \N^2}H^0(X,dL-mE).$$ This is finitely generated using the same argument as Proposition \ref{fg}.

\end{proof}
Counting the number of partitions of an integer is well-known to be a difficult problem, in the sequel we will carry out the computation in some special cases.
\end{subsection}

\medskip
\begin{subsection}{Projective bundles $\P (\O\oplus \O \oplus \O(n))$ over $\P^1$}
Keeping the notation of the previous section, take $s=1$ and $r=2$. In particular, $P(ak-i,2)=ak-i+1$. The variety $X$ now has dimension 3. In this set up, the Hilbert polynomial is
$$ h(k)= \sum_{i=0}^{ak}(ak-i+1)(bk+in+1),$$
and  so
$$a_0=\frac{1}{6}a^3n+\frac{1}{2}a^2b, \qquad a_1=\frac{1}{2}a^2n+\frac{1}{2}a^2+\frac{3}{2}ab.$$
The weight polynomial is
$$ w(k)= \sum_{i=0}^{ak}i(ak+1)(bk+in+1),$$
this gives 
$$b_0=\frac{1}{12}a^4n+\frac{1}{6}a^3b \qquad b_1=\frac{1}{3}a^3n+\frac{1}{6}a^3+\frac{1}{2}a^2b.$$

The numerator of Donaldson-Futaki invariant is
 $$\DF_{num}=-\frac{1}{72}a^6n^2+\frac{1}{72}a^6n-\frac{1}{24}a^5bn.$$
This is negative since $n^2$ is bigger than $n$. The norm is
$$\|\scF\|_2=\frac{1}{144}a^3(18a^4bn+6a^4n^2+24a^3b^2+8a^3bn-an-2b),$$
which is positive as $a$,$b$ and $n$ are positive integers.
\end{subsection}

\begin{subsection}{Projective bundle $\P( \O\oplus \O(1))$ over $\P^2$}
Let $$X=\P( \O_{\P^2}\oplus \O_{\P^2}(1)).$$
We have
$$-K_X=2(\O(1)+H).$$
To ease notation we denote $L= -\frac{1}{2}K_X$. Using the formulae of Section \ref{GeneralProjectiveBundle}, the Hilbert polynomial is
$$h(k)=\sum_{i=0}^k\binom{2+i+k}{i+k},$$
giving
$$a_0= \frac{7}{6}, \qquad a_1=\frac{7}{2}.$$

The unipotent radical of the automorphism group is $H^0(\P^2,\O(1))$. The graded modules associated to the Loewy filtration are
$$
\gr_i=H^0(\P^2,\O(i+k)).
$$
The weight polynomial is
$$w(k)=\sum_{i=0}^ki\binom{2+i+k}{i+k},$$
so
$$ b_0= \frac{17}{24}, \qquad b_1=\frac{9}{4}.$$
The numerator of the Donaldson-Futaki invariant is
$$\DF_{num}=-\frac{7}{48},$$
which is negative. The Loewy length is $\ll(k) = k$. The norm is $\|\scF\|_2=\frac{97}{1120}$. \end{subsection}

\end{section}

\bibliography{Loewy.bib}{}

\begin{thebibliography}{10}

\bibitem{ASS}
Ibrahim Assem, Daniel Simson, and Andrzej Skowro{\'n}ski.
\newblock {\em Elements of the representation theory of associative algebras.
  {V}ol. 1}, volume~65 of {\em London Mathematical Society Student Texts}.
\newblock Cambridge University Press, Cambridge, 2006.
\newblock Techniques of representation theory.

\bibitem{RB}
R.~J. {Berman}.
\newblock {K-polystability of Q-Fano varieties admitting Kahler-Einstein
  metrics}.
\newblock {\em ArXiv e-prints}, May 2012.

\bibitem{BCHM}
Caucher Birkar, Paolo Cascini, Christopher~D. Hacon, and James McKernan.
\newblock Existence of minimal models for varieties of log general type.
\newblock {\em J. Amer. Math. Soc.}, 23(2):405--468, 2010.

\bibitem{BHJ}
S.~{Boucksom}, T.~{Hisamoto}, and M.~{Jonsson}.
\newblock {Uniform K-stability, Duistermaat-Heckman measures and singularities
  of pairs}.
\newblock {\em ArXiv e-prints}, April 2015.

\bibitem{CDS}
Xiuxiong Chen, Simon Donaldson, and Song Sun.
\newblock K\"ahler-{E}instein metrics and stability.
\newblock {\em Int. Math. Res. Not. IMRN}, (8):2119--2125, 2014.

\bibitem{CDS3}
Xiuxiong Chen, Simon Donaldson, and Song Sun.
\newblock K\"ahler-{E}instein metrics on {F}ano manifolds. {III}: {L}imits as
  cone angle approaches {$2\pi$} and completion of the main proof.
\newblock {\em J. Amer. Math. Soc.}, 28(1):235--278, 2015.

\bibitem{CortiFlips}
Alessio Corti, editor.
\newblock {\em Flips for 3-folds and 4-folds}, volume~35 of {\em Oxford Lecture
  Series in Mathematics and its Applications}.
\newblock Oxford University Press, Oxford, 2007.

\bibitem{Debarre}
Olivier Debarre.
\newblock {\em Higher-dimensional algebraic geometry}.
\newblock Universitext. Springer-Verlag, New York, 2001.

\bibitem{RD}
R.~{Dervan}.
\newblock {Uniform stability of twisted constant scalar curvature K\"ahler
  metrics}.
\newblock {\em Int. Math. Res. Not. IMRN}, 2015.
\newblock To appear.

\bibitem{SD2}
S.~K. Donaldson.
\newblock Scalar curvature and projective embeddings. {I}.
\newblock {\em J. Differential Geom.}, 59(3):479--522, 2001.

\bibitem{SD3}
S.~K. Donaldson.
\newblock Scalar curvature and stability of toric varieties.
\newblock {\em J. Differential Geom.}, 62(2):289--349, 2002.

\bibitem{SD}
S.~K. Donaldson.
\newblock Lower bounds on the {C}alabi functional.
\newblock {\em J. Differential Geom.}, 70(3):453--472, 2005.

\bibitem{Z}
S.~K. Donaldson.
\newblock Stability, birational transformations and the {K}ahler-{E}instein
  problem.
\newblock In {\em Surveys in differential geometry. {V}ol. {XVII}}, volume~17
  of {\em Surv. Differ. Geom.}, pages 203--228. Int. Press, Boston, MA, 2012.

\bibitem{Gross}
Daniel Gross.
\newblock On the fundamental group of the fixed points of a unipotent action.
\newblock {\em Manuscripta Math.}, 60(4):487--496, 1988.

\bibitem{GH}
G.~Horrocks.
\newblock Fixed point schemes of additive group actions.
\newblock {\em Topology}, 8:233--242, 1969.

\bibitem{JJ}
Jens~Carsten Jantzen.
\newblock {\em Representations of algebraic groups}, volume 107 of {\em
  Mathematical Surveys and Monographs}.
\newblock American Mathematical Society, Providence, RI, second edition, 2003.

\bibitem{JL}
Z.~{Jelonek} and M.~{Laso{\'n}}.
\newblock The set of fixed points of a unipotent group.
\newblock {\em J. Algebra}, 322(6):2180--2185, 2009.

\bibitem{Levine}
Marc Levine.
\newblock A remark on extremal {K}\"ahler metrics.
\newblock {\em J. Differential Geom.}, 21(1):73--77, 1985.

\bibitem{LX}
Chi Li and Chenyang Xu.
\newblock Special test configuration and {K}-stability of {F}ano varieties.
\newblock {\em Ann. of Math. (2)}, 180(1):197--232, 2014.

\bibitem{YM}
Yoz{\^o} Matsushima.
\newblock Sur la structure du groupe d'hom\'eomorphismes analytiques d'une
  certaine vari\'et\'e k\"ahl\'erienne.
\newblock {\em Nagoya Math. J.}, 11:145--150, 1957.

\bibitem{PR}
Dmitri Panov and Julius Ross.
\newblock Slope stability and exceptional divisors of high genus.
\newblock {\em Math. Ann.}, 343(1):79--101, 2009.

\bibitem{RT}
Julius Ross and Richard Thomas.
\newblock A study of the {H}ilbert-{M}umford criterion for the stability of
  projective varieties.
\newblock {\em J. Algebraic Geom.}, 16(2):201--255, 2007.

\bibitem{JS2}
J.~{Stoppa}.
\newblock {A note on the definition of K-stability}.
\newblock {\em ArXiv e-prints}, November 2011.

\bibitem{JS}
Jacopo Stoppa.
\newblock K-stability of constant scalar curvature {K}\"ahler manifolds.
\newblock {\em Adv. Math.}, 221(4):1397--1408, 2009.

\bibitem{GS}
G.~{Sz{\'e}kelyhidi}.
\newblock {Filtrations and test-configurations}.
\newblock {\em ArXiv e-prints}, November 2011.

\bibitem{ICM}
G.~{Sz{\'e}kelyhidi}.
\newblock {Extremal K\"ahler metrics}.
\newblock {\em Proceedings of the ICM}, August 2014.

\bibitem{GT}
Gang Tian.
\newblock K-stability and {K}\"ahler-{E}instein metrics.
\newblock {\em Comm. Pure Appl. Math.}, 68(7):1085--1156, 2015.

\bibitem{DWN}
David Witt~Nystr{\"o}m.
\newblock Test configurations and {O}kounkov bodies.
\newblock {\em Compos. Math.}, 148(6):1736--1756, 2012.

\end{thebibliography}
\bibliographystyle{plain}

\vspace{4mm}

\end{document}